\documentclass[12pt]{article}
\def\myfont{\fontfamily{cmss}\selectfont}

\usepackage{fancyhdr, amsmath, amssymb, amsfonts, url, mathrsfs, graphicx, epsfig,  amsthm}
\newtheorem{thm}{Theorem}[section]
\newtheorem{cor}[thm]{Corollary}
\newtheorem{lemma}[thm]{Lemma}
\newtheorem{proposition}[thm]{Proposition}
\newtheorem{definition}[thm]{Definition}
\newtheorem{example}[thm]{Example}
\newtheorem{remark}[thm]{Remark}
\newtheorem{problem}[thm]{Problem}

\begin{document}
\date{14 Feb  2016}
 \title{Controlling the statistical properties of expanding maps}
\author{Stefano Galatolo and Mark Pollicott}
\date{Pisa University and Warwick University}
\maketitle
 \myfont

\begin{abstract}
How can one change a system, in order to change its statistical properties in a prescribed way?
In this note we consider a control  problem related to the theory of linear response.  
Given an expanding map of the unit circle with an associated invariant density we can consider the inverse problem of finding which first  order changes  in the transformation can achieve a given first order perturbation in the density.
We show the general mathematical structure of the problem, the existence of many solutions in the case of  expanding maps of the circle and the existence of optimal ones. We investigate in depth the example of the doubling map, where we give a complete solution of the problem.
\end{abstract}

\section{Introduction}

An important idea,   which has attracted much interest in recent years,  is that of linear response.
The basic principle is that in many cases a first order  change in a system leads to a corresponding  first order change in its equilibrium state.  
There has been a wealth of work on linear response theory for dynamical systems after the pioneering work of Ruelle, who  developed a formula for the  first derivative of the physical  or SRB measure for  hyperbolic systems \cite{R}. In subsequent works the approach  was simplified, applied to some system outside the uniform hyperbolic case and applied to other kinds of more or less chaotic or hyperbolic system  (see \cite{Ba2} for a detailed overview, and \cite{Ba3} for very recent results on intermittent systems).  

Many problems involving physical and social systems are modelled by chaotic dynamics and therefore many mathematical ideas have been developed and implemented in the context of  the prediction and description of the statistical behaviour of a chaotic system. 
 Linear response itself is used to understand the behaviour of  complex systems out of equilibrium (see e.g.  \cite{Lu} for an application to models of the climate evolution).

Besides merely trying to  understand  the behaviour of a given system it is also important  to attempt to understand the extent to which it can be controlled.
 For example,  {\em can one  determine which    small changes in the system will change the statistical behaviour in a prescribed direction? 
Can these changes   be  chosen  optimally, in an appropriate sense?}

Understanding the general behaviour in light of  these questions will be of help in designing efficient strategies of intervention for the control and management of complex, chaotic systems.

We shall initiate the investigation of these  specific   questions in the context of  particularly simple models of chaotic systems.  
In particular, we will concentrate on expanding maps of the circle.
We shall consider  the general mathematical structure of the problem and show that in the case of circle expanding maps, it has many solutions amongst which  we can choose an optimal one, in a suitable sense.
As an illustration, in this article we will also present a complete solution to  the problem in the particular case of  the linear doubling map, finding a solution which minimizes the $L^2$  norm (and other natural norms  on the space).

To formulate the problem more precisely, we    introduce some notation.  Let $T_0: X \to X$ be a $C^4$ expanding orientation preserving map of the circle $X = \mathbb R/\mathbb Z$ of degree $d \geq 2$. 
Let $T_\delta: X \to X$, where  $\delta \in (-\eta, \eta)$  be a  family of $C^3$ expanding maps.
Let us suppose that the dependence of the  family on $\delta$ is   differentiable at $0$, hence  can be written  
$$
T_\delta(x) = T_0(x) + \delta \epsilon(x) + o_{C^3}(\delta) \hbox{  for } x \in X.
$$
where $\epsilon \in C^3(X, \mathbb R)$, and $o_{C^3}(\delta)$ denotes a term whose $C^3$ norm tends to zero faster than $\delta$,   as $\delta \to 0$.\footnote{More precisely
we say that $T_\delta$ is a differentiable  family of $C^3$ expanding maps if there exists $\epsilon \in C^3(X, \mathbb R)$ such that 
 $\|(T_\delta - T_0) /\delta -\epsilon  \|_{C^3}  \to 0$ as $\delta \to 0$, where
 $$
  \|f(x)\|_{C^3} = \sup_{x \in X} |f(x)| + \sup_{x \in X} |f'(x)| + \sup_{x \in X} |f''(x)| + \sup_{x \in X} |f'''(x)|
 $$
    is the usual norm on $C^3$ functions.}

\begin{example}
A simple example might be where $d\geq 2 $ and $T_0(x) = dx \hbox{ (mod $1$)}$ and 
$T_\delta(x) = dx + \delta \sin(2\pi x) \hbox{ (mod $1$)}$ then trivially $\epsilon(x) = \sin(2\pi x)$.
\end{example}

For each $T_\delta$ we can associate a unique  invariant $C^2$ density and it is known that these vary in a differentiable way $\rho_\delta = \rho_0 + \delta \rho^{(1)} + \cdots$.  
It is a folklore result   that the density  of the unique physical  invariant measure of  such a family of maps varies in a differentiable way.  For a more precise statement we have  the following.
%\footnote{{\bf S}the Lemma is not explicitly used in the text, as we use the following theorem, which includes a formula for the derivative that we need. Formally the lemma tells us something a little more on the regularity of the derivative. But at the moment is not used.
%I think that in a final version of the paper we should replace this with a short recall of the story of the linear response, including the citation to this result. }.

%According to Polina this require a definition

\begin{lemma}\label{pert}
Let us assume that $T_\delta$ is a $C^1$ family of $C^3$ expanding maps.  
The density $\rho_\delta \in C^1(X,\mathbb R)$ has a continuously differentiable dependence on $\delta$.
%, i.e., there exists $\rho^{(1)} \in C^1(X, \mathbb R)$ such that 
%$$
%\rho_\delta(x) = \rho_0(x) + \delta  \rho^{(1)}(x) + o_{C^1}(\delta).
%$$
\end{lemma}

Another reformulation of the response problem would be to consider the integrals of a fixed smooth function  with respect to the varying natural  measure (see e.g. \cite{Ba2}).

 A more general  statement of Lemma \ref{pert} appears as Theorem 20 in \cite{PJ}.    The proof is omitted, but it is a standard approach using the Implicit Function Theorem for Banach manifolds.
With this notation we  can now formulate the basic problem that we want to address: 
\begin{problem} \label{pr1}{\it What  first order change $\epsilon(x)$ in $T_0$ will result in a given first order change $\rho^{(1)}$ in the density?}
\end{problem}

Before addressing this question,  we   recall some general results about linear response of systems under small perturbations.     For hyperbolic systems (including Anosov diffeomorphisms and flows) we have the luxury that this class is open which ensures that under small perturbations the system remains hyperbolic.
Moreover, such systems are structurally stable which makes it easier to apply ideas from thermodynamic formalism \cite{Contreras}.
In this context Ruelle formulated an explicit expression for the derivative of the SRB measure \cite{R} in terms of the perturbation of the system,  which was subsequently reproved using a different method in \cite{L2} \cite{GL} and \cite{Butterley-Liverani}.   
Linear response results nowadays have been proved for several classes of dynamical systems, even outside of these  uniformly expanding and hyperbolic and structurally stable  settings, see for example  \cite{DD}, \cite{BS},\cite{Ba3},\cite{BaS},\cite{Ko} and the survey \cite{Ba2}.

\noindent {\bf  Acknowledgements.} The authors thank The Leverhulme Trust for
support through Network Grant IN-2014-021.

\section{Linear response for expanding maps}

A standard tool used  in  characterising   the invariant  densities for expanding maps  is the transfer operator.  In this section we recall the definition of the transfer operator and we  present a  theorem describing  the linear response for operators satisfying certain assumptions. This tool will be directly applied to get the linear response formula for expanding maps.

\begin{definition}
Let $\mathcal  L_\delta: L^1(X, \mathbb R) \to L^1(X, \mathbb R)$ be the transfer operators associated to an expanding map $T_\delta$, defined by 
$$
\mathcal  L_\delta w(x) = \sum_{i=1}^d \frac{w(y_i^\delta)}{T_\delta'(y_i^\delta)}
$$
where the summation is over the $d$ pre-images $\{y_i^\delta\}_{i=1}^d := T_\delta^{-1}(x)$ of $x\in X$.
\end{definition}

The following is a  classical result but can be also taken as a definition of the invariant density.
\begin{definition}\label{density}
An  invariant density  $\rho_\delta $  for  $T_\delta$ is a fixed point for the operators  $\mathcal L_\delta$ acting on $L^1( X,\mathbb R)$ (i.e., $\mathcal  L_\delta \rho_\delta = \rho_\delta $).
\end{definition}

Now let us take  a general point of view and describe a general result on linear response in terms of  fixed points of operators.
Let us consider the action of these  operators $\mathcal{L}_{\epsilon }$ on different spaces. Let $B_{w},B_{s},B_{ss}$ denote abstract Banach spaces of Borel measures on $X$ equipped with norms $ ||~||_{w},||~||_{s},||~||_{ss}$ respectively, such that $||~||_{w}\leq
||~||_{s}\leq ||~||_{ss}$. We suppose that $\mathcal{L}_{\delta },$ $\delta \geq 0,$ has a unique
fixed point $h_{\delta }\in B_{ss}.$ Let $\mathcal{L:=L}_{0}$ be the
unperturbed operator and $h\in B_{ss}$ be its invariant measure. Let us consider the space of zero average measures $$
V_{s}^{0}=\{v\in B_{s},v(X)=0\}.$$

Following an approach from \cite{L2},     we present a general setting in which differentiable dependence  and {\em a formula for the derivative of the physical measure}  of a family of positive operators $\mathcal{L}_{\delta }$ can be obtained (see \cite{BGN} for a proof of the statement in this form).  
We need to consider several norms for our operators, let us denote
  $$
  \begin{aligned}
 & ||%
\mathcal{L}_{\delta }^k h||_{B_{w}\rightarrow B_{w}}   = \sup_{\| h \|_{w} \leq 1 } \|\mathcal{L}_{\delta }^k h\|_{B_w}
\cr
&||%
\mathcal{L}_{\delta }^k h||_{B_{s}\rightarrow B_{w}}   = \sup_{\| h \|_{s} \leq 1 } \|\mathcal{L}_{\delta }^k h\|_{B_w}
\hbox{ and }
&||%
\mathcal{L}_{\delta }^k h||_{V_{s}^0\rightarrow B_{s}}   = \sup_{\| h \|_{V_s^0} \leq 1 } \|\mathcal{L}_{\delta }^k h\|_{B_w}.
\end{aligned}
$$

The result we  require is the following.

\begin{thm}
\label{perturbation} Suppose that the following assumptions hold:

\begin{enumerate}
\item The norms $||\mathcal{L}_{\delta }^{k}||_{B_{w}\rightarrow B_{w}}$ are uniformly
bounded with respect to $k$ and $\delta \geq 0$.

\item $\mathcal{L}_{\delta }$ is a perturbation of $\mathcal{L}$ in the
following sense: there is a constant $C$ independent of $\delta$ such that
\begin{equation}
||\mathcal{L}_{\delta }-\mathcal{L}||_{B_{s}\rightarrow B_{w}}\leq
C\delta .  \label{approx}
\end{equation}

\item The operators $\mathcal{L}_{\delta }$, have uniform rate of contraction on $V_{s}^{0}$: there are $C_{1}>0$, $0<\rho <1$, such that $\forall \delta \in [0,1]$ 
\begin{equation}
||\mathcal{L}_{\delta }^{n}||_{V_{s}^{0}\rightarrow B_{s}}\leq C_{1}\rho
^{n}.  \label{unicontr}
\end{equation}

\item There is an operator $\mathbb {L}:B_{ss}\rightarrow B_{s}$ such
that $\forall f\in B_{ss}$
\begin{equation}
\lim_{\delta \rightarrow 0}||\delta ^{-1}(\mathcal{L}-\mathcal{L}%
_{\delta })f-\mathbb{L}f||_{s}=0.  \label{Lhat}
\end{equation}%
Let 
\begin{equation}\label{LR}
\hat{h}=(Id-\mathcal{L})^{-1}\mathbb {L}h.
\end{equation}%
Then%
\begin{equation*}
\lim_{\delta\rightarrow 0}||\delta ^{-1}(h-h_{\delta })-\hat{h}%
||_{B_w}=0;
\end{equation*}%
i.e. $\hat{h}$ represents the derivative of $h_{\delta }$ for small
increments of $\delta $.
\end{enumerate}
\end{thm}

Let us see how to apply Theorem \ref{perturbation} to our setting. The assumptions of the above  theorem are valid for the  perturbations of circle expanding maps we are considering with  the choices of spaces: 
\begin{enumerate}
\item 
$B_{w}=L^1(X)$ with the norm $\|f \|_{L^1} = \int_X |f(x)| dx$; 
\item
$B_{s}={W^{1,1}}$ with the norm $\|f \|_{W^{1,1}} =   \int_X |f(x)| dx + \int_X |f'(x)| dx$; 
\item 
$B_{ss}=C^2(X)$ with the norm $\|f \|_{C^2(X)} =  \|f\|_\infty + \|f'\|_\infty +  \|f''\|_\infty$
where $\|f\|_\infty = \sup_{x\in X} |f(x)|$.
\end{enumerate}

In this context, item $1)$ of Theorem \ref{perturbation} is trivial on $L^1$ as the transfer operators are weak contractions.  Items  $2)$ and  $3)$ of Theorem \ref{perturbation} are proved for example in \cite{G}, Section 6.
%\footnote{What are the precise  hypotheses on the perturbation $T_\delta$ needed here? {\bf S}: the assumptions of the theorem are given on the behavior of the transfer operators, but are verified for $C^2$ circle expanding maps with small $C^2$ perturbations and the spaces $W^{1,1}$ and $L^1$   in \cite{G}, probably I will explain better this point.    }
  The existence of the operator $\mathbb {L}$,  and an explicit formula  for it, will be proved in the next section (see Proposition \ref{mainprop}).
From this  follows the differentiability of the physical measure and the  linear response formula (\ref{LR}) for our  family of expanding maps.

%\footnote{Does this mean that $\rho_\delta$ is differentiable for this sort of ear perturbation, and thus we don't need Lemma 2.2 - but could change it to a weaker version?{\bf S}: I don't understand well this question, you mean Lemma 2.3? If you mean this, we dont really need it if we use Thm 2.4. {\bf S} Now  I commented on this before the Lemma}

\section{The derivative operator for circle expanding maps}

In this section we present  a detailed description of the structure of the operator $\mathbb L $ in our case.

\begin{proposition}\label{mainprop}
Let $T_\delta $ be a family of expanding maps as considered before.
Let $w \in C^3(X, \mathbb R)$.
For each $x \in X$ we can write 
\begin{equation}
\begin{aligned}
\mathbb L w(x) &= 
\lim_{\delta \to 0} \left(
\frac{\mathcal  L_\delta  w (x) - \mathcal  L_0  w(x) }{\delta}\right)\cr
&=
-\mathcal L_0\left(\frac{w\epsilon'}{T_0'}\right)(x) - 
 \mathcal L_0\left(\frac{ \epsilon w'}{T_0'}\right)(x)  + \mathcal L_0 \left(\frac{\epsilon T_0''}{(T_0')^2}  w \right)(x) 
 \end{aligned}
 \end{equation}
 and the convergence is also in the $C^1$ topology.
 %\eqno(3.1)
\end{proposition}

Before presenting  the proof of Proposition \ref{mainprop} we state  a simple  lemma.

\begin{lemma}\label{tec}
If $y_i^\delta \in T^{-1}_\delta(x)$ then we can expand 
$$y_i^\delta = y_i^0 + \delta \left( - \frac{\epsilon(y_i^0)}{T'_0 (y_i^0)}\right)
+ o_{C^2}(\delta).$$
\end{lemma}

\begin{proof}[Proof of Lemma \ref{tec}]
We denote by 
$\{y_i^\delta\}_{i=1}^d := T_\delta^{-1}(x)$ and $\{y_i^0\}_{i=1}^d := T_0^{-1}(x)$ the $d$ preimages under $T_\delta$ and $T_0$, respectively, of a point $x\in X$.
Let us write  
$$y_i^\delta(x) = y^0_i(x) + \delta \epsilon_i(x) + F_i(\delta,x).$$
We will show that $F_i(\delta,x) =  o_{C^2}(\delta)$. 
  Substituting  this into the identity $T_\delta (y_i^\delta(x) ) = x$ and 
and using that
$T_\delta(x) = T_0(x) + \delta \epsilon(x) + o_{C^3}(\delta)$ we can 
expand  
\begin{equation}
\label{long}
\begin{aligned}
x& = T_\delta( y_i^\delta(x) ) \cr
&= 
T_0( y_i^\delta(x)) + \delta \epsilon(y_i^\delta(x)) +  E(\delta, y_i^\delta(x) ) \cr
&=
T_0( y_i^0(x) + \delta \epsilon_i(x) + F_i(\delta,x)) \cr
& \  \  \  +\delta\epsilon(y_i^0(x)  +  \delta \epsilon_i(x) + F_i(\delta,x)) + E(\delta, y_i^\delta(x) ).
\end{aligned}
\end{equation}
We can write the first term in the final line of (\ref{long}) as 
$$
\begin{aligned}
&T_0( y_i^0(x) + \delta \epsilon_i (x) + F_i(\delta,x)) \cr
&= T_0( y_i^0(x)) + T_0'( y_i^0(x))(\delta \epsilon_i (x) + F_i(\delta,x)) + o_{C^2}(\delta)
\end{aligned}
$$
%\footnote{{ \bf S}:  this  $ o(\delta)$  is not depending on $x$? is it not something like  $ o(\delta \epsilon(x) + F_i(\delta,x))$? It %seems to me that the final result is ok, but I'm a bit confused about this.  {\bf M:} You are quite correct, but I expect $ o(\delta %\epsilon(x) + F_i(\delta,x)) = o_{C^3}(\delta)$ since the $o_{C^3}(\delta)$ bounds are uniform in $x$ and $T_0'$ is bounded in %$x$}
and the second term in the last line as
$$\delta\epsilon(y_i^0(x) + \delta \epsilon_i(x) + F_i(\delta,x))
= \delta\epsilon(y_i^0(x)) +\delta F_i(\delta,x)+ o_{C^2}(\delta)$$
and use that $T_0( y_i^0(x)) = x$ to cancel terms on either side of (\ref{long})  to get that 
$$0 = T_0'( y_i^0(x))(\delta \epsilon_i(x) + F_i(\delta,x))
+ \delta\epsilon(y_i^0(x)) +\delta F_i(\delta,x)
+ E(\delta, y_i^\delta(x) ) + o_{C^2}(\delta).$$
Thus we can  identify the first order terms as $\delta  T_0'( y_i^0(x)) \epsilon_i(x)  + \delta \epsilon(y_i^0(x))$ and then what is left is 
$$T_0'( y_i^0(x))F_i(\delta,x) +\delta F_i(\delta,x)= -E(\delta, y_i^\delta(x) ) + o_{C^2}(\delta)$$
from which the result follows.
\end{proof}

We now return to the proof of Proposition  \ref{mainprop}.

\begin{proof}[Proof of Proposition  \ref{mainprop}]  We again denote by 
$\{y_i^\delta\}_{i=1}^d := T_\delta^{-1}(x)$ and $\{y_i^0\}_{i=1}^d := T_0^{-1}(x)$ the $d$ preimages under $T_\delta$ and $T_0$, respectively, of a point $x\in X$.  Futhermore, we assume that the indexing is chosen so that $y^\delta_i$ is a perturbation of  $y^0_i$, for $1 \leq i \leq d$.
We can write 
\begin{equation}\label{eq1}
\begin{aligned}
&\frac{\mathcal  L_\delta  w (x) - \mathcal  L_0  w(x) }{\delta}\cr
&=
\frac{1}{\delta}
\left(
\sum_{i=1}^d \frac{w(y_i^\delta)}{T_\delta'(y_i^\delta)}
-  \sum_{i=1}^d\frac{w(y_i^0)}{T_0'(y_i^0)}
\right)\cr
&=\underbrace{ \frac{1}{\delta}
\left(
\sum_{i=1}^d w(y_i^\delta)
 \left(
 \frac{1}{T_\delta'(y_i^\delta)}
-   \frac{1}{T_0'(y_i^\delta)}\right)
\right)}_{=:(I)}
+ \underbrace{\frac{1}{\delta}
\left(
 \sum_{i=1}^d\frac{w(y_i^\delta) - w(y_i^0) }{T_0'(y_i^\delta)}
\right)}_{=:(II)}\cr
&+
\underbrace{
\frac{1}{\delta}
\left(
 \sum_{i=1}^dw(y_i^0)\left(  \frac{1}{T_0'(y_i^\delta)}- \frac{1}{T_0'(y_i^0)}\right)
\right)}_{=:(III)}.\cr
\end{aligned}
\end{equation}
 For the first term we first differentiate the expansion $T_\delta (x) = T_0(x) + \delta \epsilon(x) + o_{C^3}(\delta)$
in $x$ to get:
$$
T_\delta' (x) = T_0'(x) + \delta \epsilon'(x) + o_{C^2}(\delta).
$$
We  can then write 
$$
\begin{aligned}
(I) &=  \frac{1}{\delta}
\left(
\sum_{i=1}^d w(y_i^\delta)
 \left(
 \frac{1}{T_\delta'(y_i^\delta)}
-   \frac{1}{T_0'(y_i^\delta)}\right)
\right)\cr
&=
 \frac{1}{\delta}
\left(
\sum_{i=1}^d 
 \frac{w(y_i^\delta)}{T_\delta'(y_i^\delta)}
 \left(
 1
-   \frac{T_\delta'(y_i^\delta)}{T_0'(y_i^\delta)}\right)
\right)\cr
&=
 \frac{1}{\delta}
\left(
\sum_{i=1}^d 
 \frac{w(y_i^\delta)}{T_\delta'(y_i^\delta)}
 \left(
 1
-  \left( \frac{T_0'(y_i^\delta) + \delta \epsilon'(y_i^\delta) + o_{C^2}(\delta)}{T_0'(y_i^\delta)}\right)\right)
\right)\cr
&=
\left(
-\sum_{i=1}^d 
 \frac{w(y_i^\delta)\epsilon'(y_i^\delta)}{T_\delta'(y_i^\delta)T_0'(y_i^\delta)}\right)
+ o_{C^2}(1).
\end{aligned}
$$
Thus  we have that 
$$
\begin{aligned}
\lim_{\delta \to 0}  \frac{1}{\delta}
\left(
\sum_{i=1}^d w(y_i^\delta)
 \left(
 \frac{1}{T_\delta'(y_i^\delta)}
-   \frac{1}{T_0'(y_i^\delta)}\right)
\right)
&= 
\lim_{\delta \to 0}
\left(
-\sum_{i=1}^d 
 \frac{w(y_i^\delta)\epsilon'(y_i^\delta)}{T_0'(y_i^\delta)^2}\right) \cr
 &= -\mathcal L_0\left(\frac{w \epsilon'}{T'_0}\right)
 \end{aligned}
$$
and the limit converges in $C^1$.
For the second term  of (\ref{eq1}) we can use Lemma \ref{tec} to  write
$$
\begin{aligned}
w(y_i^\delta) &= w(y_i^0) +  w'(y_i^0)  \left(\frac{dy_i^\delta}{d\delta}|_{\delta=0} \right)\delta + o_{C^1}(\delta)\cr
&= w(y_i^0) +  w'(y_i^0) \left( - \frac{\epsilon(y_i^0)}{T'_{0}(y_i^0)}\right)
\delta + o_{C^1}(\delta).
\end{aligned}
$$
Thus 
$$
\begin{aligned}
(II) = \frac{1}{\delta} \sum_{i=1}^d\frac{w(y_i^\delta) - w(y_i^0) }{T_0'(y_i^\delta)}
 &= 
 \frac{1}{\delta} \sum_{i=1}^d\frac{w'(y_i^0) }{T_0'(y_i^\delta)}  \left( - \frac{\epsilon(y_i^0)}{T_0'(y_i^0)}\right) + o_{C^1}(1)\cr
   &= 
 -  \sum_{i=1}^d\frac{\epsilon(y_i^0)w'(y_i^0) }{T_0'(y_i^0)T_0'(y_i^\delta)} + o_{C^1}(1)
   \end{aligned}
$$
and therefore, both pointwise and in the $C^1$ topology
$$
\lim_{\delta \to 0}  \frac{1}{\delta}
 \sum_{i=1}^d\frac{w(y_i^\delta) - w(y_i^0) }{T_0'(y_i^\delta)}
 =  - \mathcal L_0\left(\frac{ \epsilon w'}{T'_0}\right)(x).
$$
Finally, for the third term we can write 
$$
\begin{aligned}
T_0'(y_i^\delta) & = T_0'(y_i^0) + T_0''(y_i^0)\left( \frac{dy^\delta_i}{d\delta}|_{\delta=0}\right) \delta  + o_{C^1}(\delta)\cr
& = T_0'(y_i^0) + T_0''(y_i^0)\left( - \frac{\epsilon(y_i^0)}{T'_{0} (y_i^0)}\right)\delta  + o_{C^1}(\delta), 
\end{aligned}
$$
again using the Lemma \ref{tec}.  
Therefore 
$$
\begin{aligned}
(III) = &\frac{1}{\delta}
\left(
 \sum_{i=1}^dw(y_i^0)\left(  \frac{1}{T_0'(y_i^\delta)}- \frac{1}{T_0'(y_i^0) }\right)
\right)\cr
&= 
\frac{1}{\delta}
\left(
 \sum_{i=1}^d w(y_i^0)\left(  \frac{T_0'(y_i^0) - T_0'(y_i^\delta)}{T_0'(y_i^\delta) T_0'(y_i^0)}\right)
\right)\cr
&= 
\frac{1}{\delta}
\left(
 \sum_{i=1}^dw(y_i^0)\left(  \frac{-\left(T_0'(y_i^0) + T_0''(y_i^0)\left( - \frac{\epsilon(y_i^0)}{T'_0(y_i^0)}\right)\delta\right) + T_0'(y_i^0)}{T_0'(y_i^\delta) T_0'(y_i^0)}\right) \right) + o_{C^1}(1)
\cr
&=\frac{1}{\delta}
\left(
  \sum_{i=1}^dw(y_i^0)\left(  \frac{\epsilon(y_i^0) T_0''(y_i^0)}{T_0'(y_i^0)^2 T_0'(y_i^\delta)}\right)\
\right) + o_{C^1}(1)\cr
\end{aligned}
$$
and thus, finally,  
$$
\lim_{\delta \to 0} \frac{1}{\delta}
\left(
 \sum_{i=1}^dw(y_i^0)\left(  \frac{1}{T_0'(y_i^\delta)}- \frac{1}{T_0'(y_i^0)}\right)
\right)
=  \mathcal L_0 \left(\frac{\epsilon T_0''}{(T_0')^2} w\right) (x) 
$$
in $C^1$.
\end{proof}

\section{The control problem}
We can now use the preceding results to address Problem \ref{pr1}, finding first order perturbations to the expanding maps corresponding to a  given first order perturbation in the density.

Putting together the information given by Theorem \ref{perturbation}   and  Proposition \ref{mainprop} in the  particular case that $w = \rho$ is the  invariant density for   $T_0: X \to X$ we  arrive at an  equation that allows us to address the question posed in Problem \ref{pr1}.   
More precisely, given the required direction of  change $\rho^{(1)} \in C^1(X, \mathbb R)$ we want to find $\epsilon(x)$ such that the associated operator $\mathbb L$ satisfies
$$
(I - \mathcal L_0) \rho^{(1)} = \mathbb L \rho(x).$$
Thus by Proposition  \ref{mainprop},  given $\rho^{(1)}$  we need to solve for $\epsilon(x)$ such that 
\begin{equation}\label{prob}
(I - \mathcal L_0) \rho^{(1)}(x) = 
 \mathcal L_0\left(-\frac{\rho \epsilon'}{T_0'}  -
\frac{ \epsilon \rho '}{T_0'}  + \frac{\epsilon T_0''}{(T_0')^2} \right)(x).
\end{equation}
The problem hence involves the solution of a differential equation for $\epsilon: X \to \mathbb R$ in the space of functions representing infinitesimal changes of the system.
We remark that given a function  $f$, the equation $f=\mathcal L_0 g $ typically may have many solutions, as $T_0$ is not bijective.

Thus in the solution of Problem \ref{pr1} there are two steps:
\begin{enumerate}
\item
Firstly we have to choose a solution
 $g$  to 
$(I - \mathcal L_0) \rho^{(1)} =  \mathcal L_0\left(g \right)$.  Typically there will be an infinite dimensional family of solutions; and  \item
Secondly, we  can solve for $\epsilon$ such that
$$g = 
-\epsilon' \left( \frac{\rho}{T_0'}\right)  - \epsilon\left(  \frac{  \rho '}{T_0'} \right) + \epsilon \left(  \frac{ T_0''}{(T_0')^2} \right). 
$$
In particular, we have to solve  a linear first order inhomogeneous  differential equation.  
\end{enumerate}

Finally, amongst the possible solutions $\epsilon(x)$ of this problem,
it  is natural to look for one which is optimal, in a suitable sense.
In particular, we will  search for a solution of minimal size with respect to some natural norm to be described later. We will prove that under suitable assumption such an optimal solution exists. 

\subsection{Existence of a solution}
Our first main conclusion about existence is the following.  

\begin{thm}\label{T1}
Let $k \geq 4$.  
For a 
$C^k$
%$C^4$ 
expanding map of the circle $T_0$,  any   
%$C^2$
$C^{k-2}$
 first order perturbation $\rho^{(1)}$ in the density can be realised by a suitable first order 
 %$C^3$
 $C^{k-1}$-perturbation $\epsilon$ in the transformation.  Moreover there is an  infinite  dimensional space of perturbations achieving this. 
\end{thm}

Before giving the proof of Theorem \ref{T1} we state a   lemma regarding the solutions of $f=\mathcal{L}_0 g$. 
In the setting of   expanding maps of  the circle this takes a simple form.
 Let $C^{k-1}(X, \mathbb R)$ be the Banach space of $C^{k-1}$ functions on the unit circle $X$, for $k \geq 1$.  

%\begin{lemma}\label{density}
%Let $T: X \to X$ be a $C^k$ expanding map\footnote{{\S}This lemma is not really used. Moreover its proof refers to the Lemma itself What we do?, we explain all better or maybe is better to delete it? } of the circle with invariant density $\rho = \mathcal L_0\rho \in C^{k-1}(X)$.
%We can write $C^{k-1}(X, \mathbb R) = E\oplus F$ where both $E$ and $F$ are preserved by   $\mathcal L_0$ and:
%\footnote{S: should they preserve it for this general remark? It seems to me that is not necessary, moreover I dont understand the notation $ C^1(X, \mathbb R)/\mathbb C$ } measure on the circle we have that:
%\begin{enumerate}
%\item 
%$E$ is a one-dimensional subspace  consisting of scalar multiples of  $\rho$; 
%\item
% $(I - \mathcal L_0) : F  \to F$ is invertible.  In particular, 
%$(I - \mathcal L_0)^{-1}F  = F$ is bijective; and
%\item
%$\mathcal L_0: C^{k-1}(X, \mathbb R) \to C^{k-1}(X, \mathbb R)$ is surjective.
%\end{enumerate}
%\end{lemma}
%
%\begin{proof}
%The splitting   $C^{k-1}(X, \mathbb R) = E\oplus F$ is a consequence of $1$ being an isolated eigenvalue for $\mathcal L_0$. 
%Part 1 follows from Lemma \ref{density}.  Part 2 comes from the fact that the operator has a spectral gap.  Part 3 is easy to see from the definition of $\mathcal L_0$.
%\end{proof}

Let $T: X \to X$ be a $C^k$ expanding map of the circle $X$.  It is known that such maps have a unique  $C^{k-1}$ invariant density.
Let $h: X \to X$ be an orientation preserving  $C^{k}$ diffeomorphism.  We can define a new map $S: X \to X$ by conjugation $S = h\circ T \circ h^{-1}$.   

We can define  transfer operators 
$\mathcal L_S:   C^{k-1}(X) \to C^{k-1}(X)$ 
and $\mathcal L_T:   C^{k-1}(X) \to C^{k-1}(X)$ 
by 
$$
\mathcal L_Tw(x) = \sum_{Ty = x}  \frac{w(y)}{T'(y)}
\hbox{ and }
\mathcal L_Sw(x) = \sum_{Sy = x}  \frac{w(y)}{S'(y)}.
$$

\begin{lemma}\label{L1}
For each  orientation preserving  $C^{k}$ diffeomorphism $h$ we can write
%footnote{Where $a.b$ denotes product between $a$ and $b$} 
$(\mathcal L_Sw)\circ h  =  \frac{1}{h'} \mathcal L_T \left((w\circ h) h'\right)$.
\end{lemma}
\begin{proof}
%(Of Lemma \ref{L1})
We can differentiate $S \circ h = h \circ  T$ and use the chain rule to  write
$$
(S'\circ h) h' = (h'\circ T) T'.
$$
Let us write $y = h(y')$   and $x' = h^{-1}(x)$ then we have 
$$
\begin{aligned}
\mathcal L_Sw(x) = \sum_{S(hy') = x}  \frac{w(y)}{S' \circ h(y')}
&= \sum_{T(y') = h^{-1}(x)}  \frac{w(hy') h'(y')}{h'(Ty') T'(y')}\cr
&=  \sum_{T(y') = x'}  \frac{(w\circ h)(y') h'(y')}{h'(Ty') T'(y')}\cr
&= \frac{1}{h'(x')} \sum_{T(y') = x'}  \frac{(w\circ h)(y') h'(y')}{T'(y')}.\cr
\end{aligned}
$$
This corresponds to the required identity.
\end{proof}

If we assume that $T$ has an invariant density $\rho$ then we know that $\mathcal L_T \rho = \rho \in C^{k-1}(X)$.  By a suitable choice of coordinates we can assume that $S(0)=0=T(0)$ are  corresponding fixed points.   We can then define $h: X \to X$ by
$h(x) = \int_0^x \rho(u) du$.  Then $h'(x) = \rho(x)$ and by the Lemma \ref{L1} with $w $ taking the constant value $1$ we can write
$$
 (\mathcal L_S1)\circ h  =  \frac{1}{h'} \mathcal L_T (1  h') = \frac{1}{\rho} \mathcal L_T (\rho) =1.
$$
We then conclude that $(\mathcal L_S1) =1$.  In particular, the transformation $S: X \to X$ has density $1$
(i.e., $S$ preserves the Haar measure on $X$).

\begin{lemma}\label{solution}
Let $T: X \to X$ be a $C^k$ map with $k \geq 2$.   For any $f \in C^{k-1}(X)$ we can find  $g\in C^{k-1}(X)$ such that $f = \mathcal L_T(g)$.   Moreover, we can find an infinite dimensional set of such solutions $g$.
\end{lemma}
\begin{proof}
%(Of Lemma \ref{solution}
Given $f \in C^{k-1}(X)$ we can first look for a solution $g \in C^{k-1}(X)$ to $f= \mathcal L_Sg$.   However, since $\mathcal L_S 1 = 1$
we see that it merely suffices to choose $g = f \circ S$, since then    
$$
\begin{aligned}
\mathcal L_S(g)(x)  &=  \mathcal L_S(f\circ S)(x) =
\sum_{Sy = x} \frac{f(Sy)}{S'(y)} = f(x)  \sum_{Sy = x} \frac{1}{S'(y)} \cr
&  = f(x)
\underbrace{ (\mathcal L_S1)(x)}_{=1} = f(x)
 \end{aligned}
 $$ 
and  the result follows.  Furthermore, 
there are clearly  uncountably many such solutions $g$.

In the general case,  we can apply the identity in the previous lemma  
 $\mathcal L_T \left((w\circ h)  \rho\right)
 =
 \rho
 (\mathcal L_Sw)\circ h$
 with the choice
$w= (f/\rho) \circ h^{-1}\circ S$.
This gives that 
$$ 
\begin{aligned}\mathcal L_T \left(( (f/\rho) \circ h^{-1}\circ S \circ h) \rho\right)
& =
 \rho
 (\mathcal L_S ((f/\rho)\circ h^{-1} \circ S))\circ h\cr
 &  = \rho  ((f/\rho) \circ h^{-1}) \circ h =  f .
 \end{aligned}
 $$
 Thus it suffices to let $g = ((f/\rho) \circ h^{-1}\circ S \circ h)  \rho$ to get the required result.
 
Given a solution $g$, every point of the set $g+ker({\mathcal L }_T)$  is again a solution. Moreover, ${\mathcal L }_T$ is an infinite dimensional space. Indeed consider the intervals $I_1,...,I_n$ where the branches of $T_i$ of $T$ are defined, then given a $C^{k-1}$ density $\rho _1$  supported on  the interior of $I_1$ it is easy to find a density $\rho _2$  supported in $I_2$ such that  ${\mathcal L }_T (\rho _1)=-{\mathcal L }_T (\rho _2).$
\end{proof}

\smallskip
\noindent
{\it Proof of Theorem \ref{T1}.}
Recall that in  equation (\ref{prob}) we have 
%\begin{equation}
$$
(I - \mathcal L_0) \rho^{(1)}(x) = 
 \mathcal L_0\left(-\frac{\rho \epsilon'}{T_0'}  -
\frac{ \epsilon \rho '}{T_0'}  + \frac{\epsilon T_0''}{(T_0')^2} \right)(x).
$$
%\end{equation}
The left hand side is a $C^{k-2}$ function, by  assumption. Since $T_0$ is $C^k$ then  we have 
$\rho \in C^{k-1}(X)$.  By Lemma \ref{solution} we can choose a $C^{k-1}$ solution $g_1$ for the step 2 described above, 
 such that $(I - \mathcal L_0) \rho^{(1)}(x) =  \mathcal L_0\left(g \right)(x)$.
 Finally this means  that the differential equation we need to solve is
$$g_1 = 
-\epsilon' \left( \frac{\rho}{T_0'}\right)  - \epsilon\left(  \frac{  \rho '}{T_0'} \right) + \left(  \frac{\epsilon T_0''}{(T_0')^2} \right), 
$$
which 
has at least $C^{k-2}$ coefficients,  and hence a $C^{k-1}$ family of solutions, proving the statement.
Each such  solution is a    solution to Problem  (\ref{pr1})  by Theorem \ref{perturbation}.
%\footnote{this proof is for the existence of the solution , for the size of the solution space we have to work more: we need more solutions for $f=\mathcal L_0 g$ }
\qed

\medskip
% much the same way we can prove the following

In particular, taking $k=5$ gives the following corollary.

\begin{cor}\label{dopo}
For a $C^5$ expanding map of the circle $T_0$,  any first order perturbation $\rho^{(1)} \in C^3(X)$ in the density can be realised by a suitable first order perturbation $\epsilon \in C^4(X)$ in the transformation.
\end{cor}

\subsection{Optimal solutions}

We now consider the problem of finding an optimal solution amongst the many different possible solutions.
 It is natural to minimize the size of the perturbation in a suitable norm. 
Since we are considering smooth dynamics and perturbations  we can choose the following 
natural Sobolev-type norm
$$ || f||_{abcd}=||f||_2+a||f'||_2+b||f''||_2+c||f'''||_2+d||f^{i v}||_2$$
for given  $a,b,c,d \geq 0 , d>0$.
With this norm, the associated  space $H^4$ of functions having $L^2$ fourth derivatives is 
a Hilbert space.
\begin{proposition}\label{min} If $\rho^{(1)} \in C^3$  and $T_0\in C^5 $
then   in the space of $H^4$  solutions  to  (\ref{prob}) there is a unique minima 
with respect to the $ || \cdot   ||_{abcd}$ norm.
\end{proposition}
\begin{proof} We  begin by  observing that the equation
 \begin{equation}\label{prob2}
(I - \mathcal L_0) \rho^{(1)} = 
 \mathcal L_0\left(-\frac{\rho \epsilon'}{T_0'}  -
\frac{ \epsilon \rho '}{T_0'}  + \frac{\epsilon T_0''}{(T_0')^2} \right)
\end{equation}
is equivalent to $g=\tilde{ L} ( \epsilon  )$,
where $g=(I - \mathcal L_0) \rho^{(1)}$  and $\tilde{ L}$ is the linear operator on the Right Hand Side  of equation (\ref{prob2}). We observe that  $\tilde{ L}$ is continuous as an operator from $H^4$ to  $L^2(X)$ and 
thus $ker(\tilde{ L}) $ is a closed space in $H^4$.  Moreover, the  space of solutions of \ref{prob2} in $H^4$ is not empty, because of Corollary \ref{dopo}.
We can therefore deduce that the  space of solutions to (\ref{prob}) is a closed affine space on which we can search for an element of  minimum  norm.  Finally, since we are in a Hilbert space there is a unique minima.
Indeed,  we can take a solution $v$ of (\ref{prob2}) and then subtract its projection on $ker(\tilde{ L}) $, this is  orthogonal to $ker(\tilde{ L}) $ and thus to the affine space of solutions $ker(\tilde{ L}) +v$. This solution necessarily  minimizes the norm.
\end{proof}

\begin{remark}
The minimal solution found in Proposition \ref{min} is an actual solution of the initial problem. Indeed let us call $\epsilon_0 $ this solution. Then
 $$\mathcal L_0\left(-\frac{\rho \epsilon_0'}{T_0'}  - \frac{ \epsilon_0 \rho '}{T_0'}  + \frac{\epsilon_0 T_0''}{(T_0')^2} \right)\in {W^{1,1}}$$ and then by Theorem \ref{perturbation} (which can be applied since the solution $\epsilon _0 $ is in $C^3$) we get
 \begin{equation}
 \rho^{(1)} = (I - \mathcal L_0)^{-1}  \mathcal L_0\left(-\frac{\rho \epsilon_0'}{T_0'}  - \frac{ \epsilon_0 \rho '}{T_0'}  + \frac{\epsilon_0  T_0''}{(T_0')^2} \right)
\end{equation}
is the required linear response associated to the first order perturbation $\epsilon _0 $.
\end{remark}

\begin{remark}The main point of the optimization procedure is an orthogonalization.  It seems that this can be efficiently implemented by an algorithm to produce the optimal solution, once an effective characterization of $ker(\tilde{ L}) $ is provided.   
\end{remark}

\section{Example: the doubling map}
In this section we  consider a simple example which can be easily  analysed using classical Fourier series.  

Let $T: X \to X$ be  the doubling map given by $T(x) = 2 x \hbox{ (mod $1$)}$
 then $T_0' = 2$, $T_0'' = 0$ and $\rho=1$ and $\rho'=0$.
Let us write the  prescribed  perturbation   $\rho^{(1)}(x) $ in the density as a Fourier series, i.e., 
$$
\rho^{(1)}(x) = \sum_{n \in \mathbb Z} a_n e^{2\pi i n x}
$$
and then observe that since
$$
\mathcal L_0 \left( e^{2\pi i n x}\right)
= 
\begin{cases}
e^{2\pi i (n/2) x} & \hbox{ if } $n$ \hbox{ is even}\\
0 & \hbox{ if } $n$ \hbox{ is odd }
\end{cases}
$$
we have that the equation we need to solve becomes 
$$
(I - \mathcal L_0) \rho^{(1)}(x) =\mathcal L_0 \left(-\frac{\epsilon ^\prime }{2} \right)
$$
where 
$$
(I - \mathcal L_0) \rho^{(1)}(x) =  \sum_{n \in \mathbb Z} (a_n - a_{2n})e^{2\pi i n x}.
$$
Moreover, given any function $f(x)$ written in the form 
$$f(x) = 
 \sum_{n \in \mathbb Z} b_n e^{2\pi i n x}
$$ 
we see that
$$
(\mathcal L_0 f)(x) = \sum_{n \in \mathbb Z} b_{2n} e^{2\pi i n x}.
$$
Comparing coefficients, for  $f(x) (= -\frac{\epsilon^ \prime (x)}{2})$ to be a solution  to 
$(I - \mathcal L_0) \rho^{(1)}=  \mathcal L_0(f)$ now  
corresponds to 
\begin{enumerate}
\item
$b_{n}
= 
a_{n/2} - a_n$ if $n$  is even;
\item 
$b_n$ have no restrictions if $n$ is odd
\end{enumerate}
and finally we see that the (infinitely many) solutions to the linear perturbation are solutions to 
$$-\epsilon'(x) = 2f(x) = 
  \sum_{n \in \mathbb Z} (2b_n) e^{2\pi i n x}.
$$  
That is
$$\epsilon(x) =  
 - \sum_{n \in \mathbb Z} \frac{b_n}{\pi i n} e^{2\pi i n x}
$$
(where the constant of integration corresponding to $b_0$ is actually zero), and  a solution is

$$\epsilon(x) =  
 - \sum_{n \in 2\mathbb Z} \frac{a_{n/2}-a_n}{\pi i n} e^{2\pi i n x} +  \sum_{n \in \mathbb Z- 2\mathbb Z} \frac{c_n}{\pi i n} e^{2\pi i n x}.
$$
For every $\{ c_i\} \in l^2$.
$$
\| \epsilon\|_2^2= 
\sum_{n \in \mathbb Z}  \frac{ |a_{2n } - a_n|^2}{\pi^2 (2n)^2} + \sum_{n \in \mathbb Z}\frac{|c_{2n+1}|^2}{\pi^2(2n+1)^2}.
$$
In particular, this is minimised when $c_{2n+1}=0$ for $n \in \mathbb Z$  and leaves us with a  distinguished solution
$$
\epsilon _0 (x) = \sum_{n} \frac{(a_n - a_{2n})}{2 \pi i n} e^{2\pi i 2n x}.
$$

Similarly for the $|| \cdot ||_{abcd} $ norm we can reason as before,  having the same 
 distinguished solution $\epsilon_0$.

\begin{example}
In the particular case of the doubling map and $\rho^{(1)} = \sin(2\pi x)$ we have that $a_1 = \frac{1}{2i}$
and $a_{-1} =- \frac{1}{2i}$ and all the other terms are zero.  Thus $2b_2 = i$ and $2b_{-2} = -i$
and all the other terms are zero.
Thus we can write
$$
\begin{aligned}
\epsilon(x)& = \frac{e^{4i\pi x}}{4\pi } +  \frac{e^{-4i\pi x}}{4\pi } + \sum_{n \in \mathbb Z} \frac{b_{2n+1}}{2\pi n} e^{2\pi i (2n+1) x}\cr
&= \frac{1}{2\pi} \cos(4\pi x) +  \sum_{n \in \mathbb Z} \frac{b_{2n+1}}{2\pi n} e^{2\pi i (2n+1) x}.
\end{aligned}
$$
If we additionally want to choose the $u(x)$ so as to minimise the $L^2$-norm then
$$
\epsilon_0 (x)= \frac{1}{2\pi} \cos(4\pi x)  
$$
 and 
$$
\|\epsilon_0 \|_2 = \sqrt{ \int_0^1 \left(  \frac{1}{2\pi} \cos(4\pi x)\right)^2dx} = \frac{\sqrt{8}}{8\pi}.
$$
\end{example}

\begin{remark}
It is natural to consider in which way the previous results could be generalized to systems with more dimensions and with contracting directions. In this case, using suitable anisotropic norms, we can have a spectral gap, and probably  the general structure of the problem remains similar, with  Proposition \ref{perturbation} applying to a suitable  space of distributions, in a way similar to that which we have seen for circle expanding maps, however the formula in \ref{mainprop}  is quite specific to the expanding case, and in the general case a suitable generalization of the derivative operator  should apply to measures and distributions.
\end{remark}

\begin{remark}
After the completion of this manuscript, B. Kloeckner (\cite{Kl}) investigated a control problem similar to the one considered here, although
in that work,  only restricted  perturbations of the system coming from a smooth conjugacy were considered. With this point of view, using ideas similar to those in \cite{Mo}, it was  proved that there is always  at least one solution of the control problem for a large class of systems preserving a smooth invariant measure.  However, the class of admissible changes used in that paper is much smaller than in the present work, and the class of solutions found is correspondingly smaller.  In particular, the optimal changes (arising from conjugacies) found in  \cite{Kl} for expanding maps  may be very different from the ones we found in the present work, as it is shown in Section 2.4 of  \cite{Kl} for the doubling map.
\end{remark}

\hfill Email: stefano.galatolo@unipi.it, M.Pollicott@warwick.ac.uk

\begin{thebibliography}{99}
\bibitem{Ba1} Baladi, V., On the susceptibility function of piecewise expanding interval maps, \textit{Comm. Math. Phys.}, (2007) 839-859.
\bibitem{Ba2} V. Baladi, Linear response, or else, \emph{ ICM proceedings Seoul 2014 ,} vol III, pp525-545

\bibitem{Ba3} V. Baladi, M. Todd. { Linear response for intermittent maps.}  arXiv:1508.02700 (to appear in Comm. Math. Phys.)

\bibitem{BS}V. Baladi and D. Smania, Linear response formula for piecewise expanding unimodal maps, \textit{Nonlinearity}, (2008) 677--711.
\bibitem{BGN} W. Bahsoun, S.  Galatolo, A. Nisoli,  A Rigorous Computational Approach to Linear Response,   arXiv:1506.08661

\bibitem{BaS} W. Bahsoun, B. Saussol Linear response in the intermittent family: differentiation in a weighted $C^0$-norm. 	arXiv:1512.01080


\bibitem{Butterley-Liverani}
O. Butterley and  C. Liverani, 
Smooth Anosov Flows: Correlation Spectra and Stability, \emph{J. Mod. Dyn.}, 1(2):301-322, 2007. 

\bibitem{Contreras}
G. Contreras,
Regularity of topological and metric entropy of hyperbolic flows.
\emph{Math. Z.} 210 (1992) 97-111. 

\bibitem{DD} D. Dolgopyat On differentiability of SRB states for partially hyperbolic systems. {\em Inv. Math. } 155 (2004) 389-449.


\bibitem{G} S. Galatolo, Statistical properties of dynamics. Introduction to the functional analytic approach,
arXiv:1510.02615

\bibitem{GL} S.Gouezel C. Liverani Banach spaces adapted to Anosov systems  {\em Erg. Th. Dyn. Sys.} 26, 1, 189--217, (2006).

\bibitem{Ko} A. Korepanov.  Linear response for intermittent maps with summable and nonsummable decay of correlations. arXiv:1508.06571 

\bibitem{Kl}  Kloeckner B. The linear request problem.  arXiv:1606.02428 

\bibitem{Mo} Moser J. On the volume elements on a manifold. Trans. AMS, 120: 286-294, 1965

\bibitem{PJ99} M. Pollicott and O. Jenkinson, Computing invariant densities and metric
entropy, \emph{Comm. Math. Phys.}  (2000), 211: 687-703.

\bibitem{PJ} O. Jenkinson and M. Pollicott. {Entropy, Exponents and invariant densities for hyperbolic systems:  Dependence and computation}, in 
\emph{Modern Dynamical Systems and its Applications} (eds. M. Brin, B. Hasselblatt, Y. Pesin),
    C.U.P., Cambridge, 2004.
    
\bibitem{Lu} V. Lucarini, F. Lunkeit, F. Ragone Predicting Climate Change using Response Theory: Global Averages and Spatial Patterns arXiv:1512.06542 [physics.ao-ph]

\bibitem{L2}C. Liverani  { Invariant measures and their properties. A functional analytic point of view}, Dynamical Systems. Part II: 
in \emph{Topological Geometrical and Ergodic Properties of Dynamics}.  Proceedings, by  Scuola Normale Superiore, Pisa (2004). 
\bibitem{R} Ruelle, D., Differentiation of SRB states, \textit{Comm. Math. Phys.} 187 (1997) 227--241.

\end{thebibliography}
\end{document}